\documentclass[12pt]{amsart}
\usepackage{amsmath}
\usepackage{amssymb}
\usepackage{amsthm}
\usepackage{enumerate}
\usepackage{url}
\usepackage{comment}
\usepackage[margin=1in]{geometry}
\newcommand{\ad}{\mathrm{ad}}
\newcommand{\adstar}{\ad^*}
\newcommand{\Ad}{\mathrm{Ad}}
\newcommand{\Adstar}{\Ad^*}
\newcommand{\Jac}{\mathrm{Jac}}
\newcommand{\transpose}{\dagger}
\newcommand{\Laplacian}{\Delta}
\newcommand{\diver}{\mathrm{div}}
\newcommand{\Diff}{\mathcal{D}}
\newcommand{\Lie}{\mathcal{L}}
\newtheorem{thm}{Theorem}[section]
\newtheorem{prop}[thm]{Proposition}
\newtheorem{lem}[thm]{Lemma}

\newtheorem*{definition}{Definition}
\newtheorem{example}[thm]{Example}
\title{Geometry of the Contactomorphism group}
\author{Boramey Chhay and Stephen C. Preston}
\address{University of Colorado}
\email{boramey.chhay@colorado.edu}
\email{stephen.preston@colorado.edu}

\date{\today}

\begin{document}

\begin{abstract}
In this paper we examine the Riemannian geometry of the group of contactomorphisms of a compact contact manifold.
We compute the sectional curvature of $\mathcal{D}_\theta(M)$ in the sections containing the Reeb field and show
that it is non-negative. We also solve explicitly the Jacobi equation along the geodesic corresponding to the flow
of the Reeb field and determine the conjugate points. Finally, we show that the Riemannian exponential map is a
non-linear Fredholm map of index zero.
\end{abstract}

\maketitle

\section{Introduction}
Let $M$ be an orientable compact manifold (without boundary) of odd dimension $2n+1$. Recall that $M$ is called a
\emph{contact manifold} if there is a $1$-form $\theta$ on $M$ satisfying
the non-degeneracy condition that $\theta\wedge d\theta^n\neq0$ everywhere. If $\mathcal{D}(M)$ denotes the group of diffeomorphisms of $M$,
we say that $\eta\in\mathcal{D}(M)$ is a contactomorphism if $\eta^*\theta$ is some positive functional multiple of $\theta$;
the group of contactomorphisms is denoted by $\mathcal{D}_{\theta}(M)$.
Keeping track of this multiple we may write $\eta^*\theta=e^\Sigma \theta$ where $\Sigma$ is some function $\Sigma \colon M\rightarrow \mathbb{R}$,
and we define the group of ``padded contactomorphisms'' to be the group
$$ \widetilde{\mathcal{D}}_{\theta}(M) = \{ (\eta, \Sigma) \, \vert \, \eta^*\theta = e^{\Sigma}\theta\}.$$
For details on these constructions, see \cite{EP}.

We will be working primarily on the Lie algebras of these groups, and we will use the following well-known
fact that the Lie algebra  $T_e\mathcal{D}_{\theta}(M)$ can be identified with the space of functions $f\colon M\to \mathbb{R}$.

\begin{prop}[EP] The Lie algebra $T_e\mathcal{D}_\theta(M)$ consists
of vector fields $u$ such that $\mathcal{L}_u\theta=\lambda \theta$ for some function
$\lambda:M\rightarrow \mathbb{R}$. Any such field is uniquely determined by the function
$f=\theta(u)$, and we write $u=S_\theta f$. Thus we have that
\[T_e\mathcal{D}_\theta(M)=\{S_\theta f:f\in C^\infty(M)\}.\]
\end{prop}


Here we call $S_\theta$ the contact operator. The Lie bracket on $T_e\mathcal{D}_\theta(M)$ is given by
\begin{equation}\label{contactbracket}
[S_\theta f,S_\theta g]=S_\theta \{f,g\},\text{    where   }\{f,g\}=S_\theta f(g)-gE(f);
\end{equation}
here $E$ denotes the Reeb vector field, uniquely specified by the conditions $\theta(E)=1$, $\iota_Ed\theta=0$.
We call $\{\cdot,\cdot\}$ the ``contact Poisson bracket''; it is not a true Poisson bracket
since it does not satisfy Leibniz's rule. 

We also need a Riemannian structure on $(M,\theta)$, and we will require
that the Riemannian metric be \emph{associated} to the contact form. It will also be convenient to assume that $E$ is a Killing field (i.e., its flow consists of isometries).

\begin{definition}
If $(M,\theta)$ is a contact manifold and $E$ is the Reeb field, a Riemannian metric $(\cdot,\cdot)_g$ is \emph{associated}
if it satisfies the following conditions:
\begin{enumerate}
\item $\theta(u)=(u,E)_g$ for all $u\in TM$, and
\item there exists a $(1,1)$-tensor field $\phi$ such that $\phi^2(u)=-u+\theta(u)E$
and $d\theta(u,v)=(u,\phi v)_g$ for all $u$ and $v$.
\end{enumerate}
If in addition $E$ is a Killing field, we say that that $(M,\theta,g)$ is \emph{$K$-contact}.
\end{definition}

Now if we have a $K$-contact manifold $(M,\theta)$ with an associated metric $(\cdot,\cdot)_g$,
we define a right-invariant metric $\langle\cdot,\cdot\rangle$, on $\mathcal{D}_\theta(M)$  by
\[\langle S_\theta f,S_\theta g\rangle=\int_M(S_\theta f,S_\theta g)_g + E(f)E(g) \, d\mu=\int_M(f-\Delta f)gd\mu,\]
where the latter formula applies since the metric is associated~\cite{EP}. This is the natural metric
induced on $\widetilde{\mathcal{D}}_{\theta}(M)$ as a submanifold of the semidirect product
$\mathcal{D}(M) \ltimes C^{\infty}(M,\mathbb{R})$.


On any Lie group with a right-invariant Riemannian metric, the geodesic equation can
be written in terms of the flow equation \[\frac{d\eta}{dt} = u\circ \eta\] and
the Euler-Arnold equation, given by
\[\frac{du}{dt}+\adstar_uu=0.\]
For $\mathcal{D}_\theta(M)$ with an associated Riemannian metric, the Euler-Arnold equation is given by
\begin{equation}\label{contacto}
(f-\Delta f)_t+S_\theta f(f-\Delta f)+(n+2)(f-\Delta f)E(f)=0.
\end{equation}

\begin{example}
\normalfont
In the case of the circle $M=S^1$ with coordinate $\alpha$, the standard $1$-form is $\theta=d\alpha$, and every diffeomorphism
is a contactomorphism.
The Reeb field is given by $\frac{\partial}{\partial \alpha}$, and the contact operator is given by $S_\theta f=fE$.
Hence the Euler-Arnold equation is given by
\[(f-f_{\alpha\alpha})_t+f(f_\alpha-f_{\alpha\alpha\alpha}) +2(f-f_{\alpha\alpha})f_\alpha=0,\]
which is the Camassa-Holm equation~\cite{CH, K, KLMP}.
\end{example}
Here we prove three results. First, we demonstrate that the flow of the Reeb field is a geodesic, and we show that the
sectional curvature is non-negative in all sections containing the Reeb field.
It is then natural to ask whether there are conjugate points along the corresponding Reeb flow geodesic. We compute the
Jacobi fields along this geodesic explicitly and find all such conjugate points. Having obtained conjugate points, it is
natural to ask whether such points must be isolated and of finite order; we prove the answer is affirmative by showing
that the differential of the exponential map is Fredholm. For simplicity of exposition we demonstrate only ``weak''
Fredholmness, though we show how one can use the technique of \cite{MP} to demonstrate strong Fredholmness in the context
of Sobolev manifolds.
\\
\\
\textbf{Acknowledgements.} The authors gratefully acknowledge the support of Simons Foundation Collaboration Grant $\#$318969.

\section{Sectional Curvature}

The curvature of a Lie group $G$ with right-invariant metric
in the section determined by a pair of vectors $X,Y$ in the Lie algebra $\mathfrak{g}$ is given by the following formula~\cite{AK}.

\begin{equation}\label{arnoldformula}
C(X,Y)=\langle d,d\rangle+2\langle a,b \rangle-3\langle a,a\rangle-4\langle B_X,B_Y\rangle
\end{equation} where
\begin{multline*}
2d=B(X,Y)+B(Y,X), \qquad 2b=B(X,Y)-B(Y,X), \\
2a=\ad_XY,\qquad 2B_X=B(X,X),\qquad 2B_Y=B(Y,Y),
\end{multline*}
where $B$ is the
bilinear operator on $\mathfrak{g}$ given by the relation $\langle B(X,Y),Z\rangle=\langle X,\ad_YZ\rangle$, i.e.,
$B(X,Y) = \adstar_YX$.
Note that in terms of the usual Lie bracket of vector fields, we have $\ad_XY = -[X,Y]$; see \cite{AK}.
The sectional curvature is given by the normalization $K(X,Y) = C(X,Y)/\lvert X\wedge Y\rvert^2$, but since we only
care about the sign, we will work with $C$ instead of $K$.

In this section we will show that the curvature takes on both signs; in fact we will show that $C(E,Y)\ge 0$
for all $Y$, and that there are many sections such that $C(X,Y)< 0$.

\begin{lem}\label{adstarlemma}
If $X = S_{\theta}f$ and $Y = S_{\theta}g$, then
\begin{equation}\label{adstar}
\adstar_XY = S_{\theta}(1-\Laplacian)^{-1}\big[ S_{\theta}f(g-\Laplacian g) + (n+2) E(f) (g-\Laplacian g)\big].
\end{equation}
\end{lem}

\begin{proof}
Let $Z = S_{\theta}h$ for some function $h$. Then we have
\begin{align*}
\langle \adstar_XY, Z\rangle &= \langle Y, \ad_XZ\rangle = -\int_M \langle S_{\theta}g, S_{\theta}\{f,h\} \rangle \, d\mu \\
&= -\int_M (g-\Laplacian g) \{f,h\} \, d\mu = -\int_M (g-\Laplacian g) \big( S_{\theta}f(h) - hE(f)\big) \, d\mu \\
&= \int_M h \big[ S_{\theta}f (g-\Laplacian g) + (g-\Laplacian g) \big(\diver(S_{\theta}f) + E(f)\big)\big] \, d\mu.
\end{align*}
Now using the fact from \cite{EP} that $\diver{(S_\theta}f) = (n+1) E(f)$ for an associated metric, we obtain
$$ \langle \adstar_XY,Z\rangle = \int_M h\big[ S_{\theta}f(g-\Laplacian g) + (n+2) E(f) (g-\Laplacian g)\big] \, d\mu.$$
Since this is true for every $h$, we conclude formula \eqref{adstar}.
\end{proof}

Combining Lemma \ref{adstarlemma} with the general formula \eqref{arnoldformula}, we obtain the following formula.

\begin{thm}
Suppose $M$ is a contact manifold with associated Riemannian metric and a Killing Reeb field. Then the sectional curvature of $\mathcal{D}_\theta(M)$ is non-negative when one of the directions is the Reeb Field.
\end{thm}

\begin{proof}
Write $X = S_{\theta}(1) = E$ and $Y = S_{\theta}g$ for some function $g$. Then we have
$\ad_XY = -[X,Y] = -S_{\theta}\{f,g\} = -S_{\theta}(E(g))$ using \eqref{contactbracket}.

Furthermore, we find that
\begin{align*}
B(X,Y) &= \adstar_YX = S_{\theta}(1-\Laplacian)^{-1}\big[ (n+2) E(g)\big], \\
B(Y,X) &= \adstar_XY = S_{\theta}(1-\Laplacian)^{-1}\big[ E(g-\Laplacian g)\big],
\end{align*}
and we conclude that $B(X,X)=0$.
%
%
%
%
%
%
%
%
Note that
there is no need to calculate $B(Y,Y)$ as it only appears in the curvature formula
coupled with $B(X,X)=0$.


Formula \eqref{arnoldformula} yields
\begin{align*}
C(X,Y) &=
\tfrac{1}{4} \langle S_\theta E(g)+S_\theta (1-\Laplacian)^{-1}[(n+2)E(g)],S_\theta E(g)+S_\theta (1-\Laplacian)^{-1}[(n+2)E(g)]\rangle\\
&\qquad\quad+\tfrac{1}{2} \langle S_\theta E(g), S_\theta E(g)-S_\theta (1-\Laplacian)^{-1}[(n+2)E(g)]\rangle-\tfrac{3}{4}
\langle S_\theta E(g),S_\theta E(g)\rangle,
\end{align*}
and thus we have that
\[C(X,Y)=\tfrac{(n+2)^2}{4} \lvert  S_\theta (1-\Laplacian)^{-1}[E(g)] \rvert^2.\]
In particular we have that $K(X,Y)$ is non-negative.
\end{proof}

Observe that the sectional curvature $K(E,S_{\theta}g)$ is zero if and only if $E(g)\equiv 0$. If this is the case, $S_{\theta}g$
actually preserves the contact form (not just the contact structure); that is, if $\eta$ is the flow of $S_{\theta}g$ then
$\eta^*\theta = \theta$, and $\eta$ is called a quantomorphism.

It would be interesting to determine whether there are any other velocities $X\in T_e\Diff_{\theta}(M)$
for which $C(X,Y)\ge 0$ for all $Y$. On the volumorphism group $\Diff_{\mu}(M)$ of a manifold $M$ of dimension
two or higher, for example, this is only true when $X$ is a  Killing field; see \cite{KLMP}.

To demonstrate negative curvature, it is sufficient to work on the quantomorphism group, using the following
result from \cite{EP}.

\begin{thm}[EP]
If $M$ is a contact manifold with a regular Reeb field $E$ for which the orbits are all closed and of the same length,
then the group of quantomorphisms $\Diff_q(M)$ consisting of those contactomorphisms $\eta$ such that
$\eta^*\theta=\theta$ is a closed and totally geodesic submanifold.
\end{thm}

As a consequence, the second fundamental form $B(X,Y)$ of the quantomorphism group in the contactomorphism group is zero,
which means by the Gauss-Codazzi formula that
$$ C_q(X,Y) = C_{\theta}(X,Y) + \langle B(X,X),B(Y,Y)\rangle - \langle B(X,Y),B(X,Y)\rangle = C_{\theta}(X,Y)$$
whenever $X$ and $Y$ are tangent to the quantomorphism group; that is, $X=S_{\theta}f$ and $Y=S_{\theta}g$ where
$E(f)=E(g)=0$. in other words, the curvature can be computed using only the quantomorphism group formulas, and these
were worked out by Smolentsev~\cite{S}.

\begin{thm}[S]
Let $M$ be a contact manifold with a regular contact form; then the group of quantomorphisms has tangent space
$ T_{e}\Diff_q(M) = \{ S_{\theta}f \, \vert \, E(f)\equiv 0\}$, and the curvature in the section spanned by $X=S_{\theta}f$
and $Y=S_{\theta}g$ is given by the formula
\begin{multline}\label{quantocurvature}
C_q(X,Y) = \tfrac{1}{4} \int_M \{f,g\}^2 \, d\mu + \tfrac{3}{4} \int_M \{f,g\} \Laplacian\{f,g\} \, d\mu + \tfrac{1}{2} \int_M \{f,g\}\big( \{f,\Laplacian g\} - \{g,\Laplacian f\}\big) \, d\mu \\
- \int_M \{f,\Laplacian f\} (1-\Laplacian)^{-1} \{g,\Laplacian g\} \, d\mu \\
+ \tfrac{1}{4} \int_M \big( \{f,\Laplacian g\} + \{g,\Laplacian f\}\big) (1-\Laplacian)^{-1}\big( \{f,\Laplacian g\} + \{g,\Laplacian f\}\big) \, d\mu.
\end{multline}
\end{thm}

Note that since the functions $f$ and $g$ are Reeb-invariant, we may equivalently think of them as defined on the 
Boothby-Wang quotient $N$, and use the Laplacian on $N$ instead of the one on $M$. Note also that Smolentsev uses
the opposite sign convention for the Laplacian $\Laplacian$.

\begin{example}
\normalfont Let $f_k$ and $g_k$ be two distinct eigenfunctions of the Laplacian on $N$ which share the eigenvalue $-\lambda_k$, and let $h_k=\{f_k,g_k\}$. Smolentsev's formula \eqref{quantocurvature} then reduces to
\[
C_q(S_\theta f_k,S_\theta g_k)=\tfrac{1}{4}\int_Mh_k^2d\mu+\tfrac{3}{4}\int_M h_k\Laplacian h_k d\mu-\lambda_k\int_M h_k^2d\mu
\]
using the anti-symmetry of the bracket. Letting $-\lambda$ be an upper bound for the eigenvalues of the Laplacian we get the inequality
\[C_q(S_\theta f_k,S_\theta g_k)\leq\left(\tfrac{1}{4}-\tfrac{3}{4}\lambda- \lambda_k\right)\int_M h_k^2d\mu.\]
Thus we get that the curvature in the section spanned by $X=S_\theta f_k$ and $Y=S_\theta g_k$ is negative whenever $\tfrac{1}{4}-\tfrac{3}{4}\lambda- \lambda_k< 0.$
For example if $M=S^3$ with the round metric of radius $2$, then the Boothby-Wang quotient is $N=S^2$ with the round metric
of radius $1$. All eigenvalues of the Laplacian on $S^2$ are of the form $n(n+1)$ for an integer $n$ with multiplicity 
at least two, and thus this construction always gives infinitely many sections of negative curvature on $\Diff_{\theta}(S^3)$.
\end{example}

\section{Conjugate Points}

It is easy to see that the function $f\equiv 1$ is a steady solution of the Euler-Arnold
contactomorphism equation \eqref{contacto}, and hence the flow of the Reeb field $E$ is a geodesic in $\Diff_{\theta}(M)$, with non-negative curvature in every section containing it.
It is natural to ask whether there are conjugate points along this geodesic. We answer this  question by solving the Jacobi equation explicitly and locating all the conjugate points; we find that they are all monoconjugate of finite order, an illustration of the fact that the exponential map is Fredholm (as we will show
in the next section).

Recall that the exponential map $\exp_p$ on a Riemannian manifold $\mathcal{M}$ at a point $p\in\mathcal{M}$
is the map $\exp_p(v) = \gamma(1)$ where $\gamma$ is the geodesic such that $\gamma(0)=p$ and $\gamma'(0)=v$. Its
differential determines the conjugate points on $\mathcal{M}$: a point $q$ is called conjugate to $p$ along $\gamma$
if $q=\gamma(\tau)$ for some $\tau$ and if $(d\exp_p)_{\tau \gamma'(0)}$ is not invertible as a map from $T_p\mathcal{M}$
to $T_q\mathcal{M}$. The exponential map is called Fredholm if $(d\exp_p)_v$ is a Fredholm linear operator for every $p$ and $v\in T_p\mathcal{M}$;
in this case the map is invertible if and only if it is one-to-one, and the nullspace is finite-dimensional.
The following proposition ends up being the most convenient way to both compute conjugate points and to prove Fredholmness
if $\mathcal{M}$ is a Lie group with right-invariant Riemannian metric.

\begin{prop}\cite{EMP, MP}\label{Fredholmprop} Suppose we have a Lie group $G$ with a right-invariant metric
and a smooth geodesic $\eta(t)$ with $\eta(0)=e$ and $\dot{\eta}(0)=u_0$. Define linear operators $\Lambda(t)$ and $K_{u_0}$ on $T_eG$  by the formulas
\[\Lambda(t)(v)=\Adstar_{\eta(t)}\Ad_{\eta(t)}(v)\]
and
\[K_{u_0}(v)=\adstar_v u_0.\]
Then the Jacobi equation solution operator \[\Phi(t)=tdL_{\eta(t)^{-1}}(d\exp_{e})_{tu_0}\]
satisfies the equation
\begin{equation}\label{Phieq}
\Phi(t)=\Omega(t)+\int_0^t \Lambda(\tau)^{-1}K_{u_0}\Phi(\tau)d\tau
\end{equation}
where
\[\Omega(t)=\int_0^t\Lambda(\tau)^{-1}d\tau.\]
\end{prop}

\begin{proof}
The proof follows from rewriting the Jacobi equation using left-translation: if $J(t)$ is a
Jacobi field and we write $J(t) = dL_{\eta(t)} w(t)$, then $w(t)$ satisfies the equation
\begin{equation}\label{leftjacobi}
\frac{d}{dt}\left( \Lambda(t) \, \frac{dw}{dt}\right) + \adstar_{dw/dt} u_0 = 0.
\end{equation}
With initial conditions $w(0) = 0$ and $w'(0) = v_0$, we write $w(t) = \Phi(t)(v_0)$
and find that $\Phi$ satisfies \eqref{Phieq}.
\end{proof}

This proposition shows that a point $\eta(\tau)$ is conjugate to the identity if and only if
$\Phi(\tau)$ is non-invertible. The operator $\Phi$ is particularly easy to compute in the
case when the curve $\eta$ is a family of isometries of the underlying manifold. First we demonstrate
that the operator $K_{u_0}$ is compact on the contactomorphism manifold.

\begin{prop}\label{Kcompact}
Suppose $M$ is a contact Riemannian manifold with associated Riemannian metric. Then for any fixed
$f$ with $u=S_{\theta}f$, the operator $K_{u}$ is compact. Since it is also anti-selfadjoint, it has a
basis of complex eigenvectors $v_k=S_{\theta}g_k$ such that $K_u(v_k) = i\lambda_k v_k$, with $\lambda_k\in\mathbb{R}$
and $\lambda_k\to 0$ as $k\to\infty$.
\end{prop}

\begin{proof}
From Lemma \ref{adstarlemma} we conclude that with $u=S_{\theta}f$ and $v=S_{\theta}g$,
\begin{equation}K_u(S_{\theta}g) =
S_{\theta}(1-\Laplacian)^{-1}\big[ S_{\theta}g (\phi) + (n+2) E(g) \phi\big],\end{equation}
where $\phi = f-\Laplacian f$.
From this equation we see that $K_u$ gains two derivatives from the inverse Laplacian
$(1-\Delta)^{-1}$ but only loses one derivative because of the contact operator $S_\theta$,
overall gaining a derivative. Thus $K_u$ is a compact operator.

The fact that $K_u$ is anti-selfadjoint follows from the equation
$$\langle K_uv,v\rangle = \langle \adstar_vu, v\rangle = \langle u, \ad_vv\rangle = 0,$$
for any $v$. Finally the statement about eigenvalues follows from the fact that $iK_u$ is
a self-adjoint compact operator and general spectral theory.
\end{proof}

\begin{thm}\label{conjugatepoints}
Suppose $M$ is a contact manifold with an associated Riemannian metric and a regular Reeb field $E$ that is also a Killing field.
Let $\eta(t)$ be a geodesic on $\mathcal{D}_\theta(M)$ with initial condition $\eta(0)=e$ and $\dot{\eta}(0)=E$.
Then $\eta(T)$ is conjugate to $\eta(0)$ for $T>0$ if and only if
$$ T = \frac{2\pi m}{\lvert \lambda\rvert},$$
where $\lambda$ is one of the real eigenvalues of $iK_E$ as in Proposition \ref{Kcompact}.
\end{thm}

\begin{proof}
Since $E$ is a steady solution of \eqref{contacto}, its flow $\eta(t)$ is a geodesic in $\Diff_{\theta}(M)$,
and since $E$ is a Killing field, every $\eta(t)$ is an isometry of $M$. We therefore have that the operator
$\Lambda(t)$ defined in Proposition \ref{Fredholmprop} is the identity, since we have for every vector field $v$ on $M$ that
$$ \langle \Lambda(t) v, v\rangle = \lvert \Ad_{\eta(t)} v\rvert^2 = \int_M \lvert D\eta(t)(v)\rvert^2\circ\eta(t)^{-1} \, d\mu.$$
The fact that $\eta(t)$ is an isometry implies that $\lvert D\eta(t)v_x\rvert^2 = \lvert v_x\rvert^2$ for every point $x\in M$,
and in addition that the Jacobian of $\eta$ is one, so that $\langle \Lambda(t)v, v\rangle = \langle v,v\rangle$ for every
vector field $v$, and hence $\Lambda(t)$ is the identity for all $t$.

Since $K_E$ is diagonalizable by Proposition \ref{Kcompact},
equation \eqref{leftjacobi} diagonalizes as well:
if $K_{u_0}(v_k) = i\lambda_k v_k$ for some $k$ with $\lambda_k\in\mathbb{R}$, then equation
\eqref{leftjacobi} for $w(t) = f(t) v_k$ where $f\colon \mathbb{R}\to \mathbb{C}$ takes the form
$ f''(t) + i\lambda f'(t) = 0,$ whose solution with $f(0)=0$ and $f'(0)=1$ is obviously
$$ f(t) = \frac{i}{\lambda} (e^{-i\lambda t} - 1).$$ We therefore get a conjugate point at time
$T = 2\pi/\lvert \lambda\rvert$, and at all integer multiples thereof.
\end{proof}

In general we can write
$$ K_E(S_{\theta}g) = (n+2) S_{\theta}(1-\Laplacian)^{-1}E(g).$$
\begin{example}
\normalfont On the $3$-sphere where $E$ is a left-invariant vector field, the
operators $\Laplacian$ and $\Lie_E:= g\mapsto E(g)$ commute (since $E$ is Killing) and have
a basis of simultaneous eigenfunctions $g_{pq}$ such that $\Laplacian g_{pq} = -q(q+2) g_{pq}$
and $\Lie_E(g_{pq}) = ipg_{pq}$ whenever $q$ is a positive integer and
$p$ is an integer in the set $\{-q,-q+2, -q+4,\cdots, q-4,q-2,q\}$; see \cite{P}.
In this case we get
$$ K_E(S_{\theta}(g_{pq})) = \frac{3ip}{(q+1)^2} S_{\theta}g_{pq},$$
%
%
%
%
and we obtain conjugate points along the Reeb geodesic at times
$ T = \frac{2\pi m(q+1)^2}{3p}$ for $q$ any positive integer, $p$ any positive integer with $p\le q$ and $q-p$ even, and $m$ any positive integer.
\end{example}
%

\section{The Exponential Map}

Now we would like to show Fredholmness of the exponential map. In some sense the three-dimensional contactomorphism equation is a hybrid of the Camassa-Holm equation and the two-dimensional Euler equation for ideal fluids as discussed in \cite{EP}, and both of these diffeomorphism groups have strongly Fredholm exponential maps~\cite{EMP, MP}. Hence intuitively we would expect the same on the contactomorphism group. To prove this, we use Proposition \ref{Fredholmprop}.

%

The point  is that by Proposition \ref{Fredholmprop}, we can essentially decompose the differential of the exponential map
into the sum of operators $\Omega(t)$ (which is positive-definite and thus invertible)
and a remainder expressed as a composition with $K_{u_0}$. Since $K_{u_0}$ is compact
by Proposition \ref{Kcompact}, we know $(d\exp_e)_{tu_0}$ will be a Fredholm operator, and hence the exponential
map will be a non-linear Fredholm map. As a consequence~\cite{MP} we obtain that conjugate points
are of finite multiplicity and form a discrete set along any geodesic, along with various other
analogues of theorems in global Riemannian geometry which would otherwise fail in infinite dimensions.

\begin{thm}
The Riemannian exponential map on $\mathcal{D}_\theta(M)$ is weakly Fredholm; that is, the differential of the exponential map extends to a Fredholm operator in the closure of $T_e\mathcal{D}_{\theta}(M)$ in the $L^2$ topology generated by the Riemannian metric.
\end{thm}

\begin{proof}
Using the notation from Proposition \ref{Fredholmprop}, we would like to show that the solution operator $\Phi$ is the sum of an invertible operator and a compact operator, thus making the exponential map Fredholm. By Proposition \ref{Kcompact}, we know $K_{u_0}$ is compact, so we just need to know that $\Lambda(t)$ is 
invertible for every $t$. It is sufficient to show that $\Lambda(t)$ is positive-definite.
 
For $v=S_\theta g$ and $\eta(t)$ a geodesic, we have that $\eta(t)^*\theta=e^{\Sigma(t)}\theta$ for some $\Sigma(t)\colon M\rightarrow \mathbb{R}$, and we compute that
\[\Ad_\eta S_\theta g=S_\theta\left( (e^\Sigma g)\circ\eta^{-1}\right).\]

Now we have that
\begin{equation}\label{Adstar}
\begin{split}
\langle v,\Adstar_{\eta(t)}\Ad_{\eta(t)}v \rangle&=\langle \Ad_{\eta(t)}v,\Ad_{\eta(t)}v \rangle\\
&=\int_M (1-\Delta)\left((e^\Sigma g)\circ\eta^{-1}(t)\right)\left((e^\Sigma g)\circ\eta^{-1}(t)\right)\,d\mu\\
&=\int_M \left(e^\Sigma g\right)^2\circ\eta^{-1}(t)d\mu+\int_M \left\vert \nabla\cdot \left((e^\Sigma g)\circ\eta^{-1}(t)\right) \right\vert^2\, d\mu\\
&=\int_M \left(e^\Sigma g\right)^2\Jac(\eta(t))d\mu+\int_M \left\vert D\eta^{-1}(t)\circ\eta(t)\nabla\cdot (e^\Sigma g) \right\vert^2 \Jac(\eta(t)) \,d\mu\\
\end{split}
\end{equation}
where the second to last line is justified by integration by parts and the last line is justified by a change of variables.

Now consider $D\eta^{-1}(t)\circ\eta(t)=(D\eta(t))^{-1}$, since we would like to bound the quantity
 $\langle v,\Adstar_{\eta(t)}\Ad_{\eta(t)}v \rangle/\langle v,v\rangle$ uniformly below by some positive number. In order to do this, we look at the eigenvalues of $D\eta^{\transpose}D\eta(t)$ and take the supremum over all of $M$; we will call this supremum $\alpha(t)$, which is finite since $M$ is compact. Thus we have
$$\int_M \left\vert D\eta^{-1}(t)\circ\eta(t)\nabla\cdot (e^\Sigma g) \right\vert^2\Jac(\eta(t)) \,d\mu \ge
\frac{1}{\alpha(t)}\int_M \left\vert \nabla\cdot (e^\Sigma g) \right\vert^2\Jac(\eta(t)) d\mu.$$
Finally using the fact that $\Jac(\eta(t))=e^{(n+1)\Sigma(t)}$ by \cite{EP} and spatial smoothness of $\Sigma(t)$,
we obtain the desired lower bound in terms of the infima of $\Sigma(t)$ and $\alpha(t)$, which depend on the $C^1$
norm of $\eta$.

Thus we get for each $t$, we have that $\langle v,\Adstar_{\eta(t)}\Ad_{\eta(t)}v \rangle/\langle v,v\rangle $ is bounded below by a positive number independent of $v$. Now integrating in time we see that $\Omega(t)$ is also positive-definite, thus invertible.
\end{proof}



We can prove strong Fredholmness using essentially the same techniques as \cite{MP}: approximate $\eta \in \Diff^s_{\theta}(M)$ by $\tilde{\eta}\in\Diff_{\theta}(M)$ (i.e., a $C^{\infty}$ geodesic with initial velocity $\tilde{u}_0$) so that Proposition \ref{Fredholmprop} (which loses derivatives) makes sense in $H^s$. Then apply commutator estimates to show that $\Lambda(t)$ is the sum of a positive-definite and a compact operator, so that $\Omega(t)$ is as well, and conclude that $(d\exp_e)_{t\tilde{u}_0}$ is a Fredholm operator. Then the fact that Fredholm operators are open in the space of all operators implies that $(d\exp_e)_{tu_0}$ is also Fredholm for $u_0\in H^s$ sufficiently close to $\tilde{u}_0$. We omit the details, which are very similar to those of \cite{MP} due to the fact that
\eqref{Adstar} is so similar to the corresponding operator for the Camassa-Holm equation.

\end{document}